\documentclass[a4paper,12pt]{amsart}
\usepackage{pst-node}
\usepackage{amsmath,amsthm,amssymb,a4wide}
\usepackage[arrow,matrix]{xy}
\usepackage{hyperref}
\usepackage{ulem}
\usepackage{graphicx}
\newgray{lightgray}{0.9}

\theoremstyle{ams}
\newtheorem{theorem}{Theorem}[section]

\newtheorem{lemma}[theorem]{Lemma}

\theoremstyle{definition}

\numberwithin{figure}{section} \numberwithin{equation}{section}

\newcommand{\DHfunction}{\mathrm{DH}}

\newcommand{\Q}{\mathbb{Q}}
\newcommand{\R}{\mathbb{R}}

\begin{document}
\title{Log-concavity of complexity one Hamiltonian torus actions}

\author[Y. Cho]{Yunhyung Cho}
\address{School of Mathematics, Korea Institute for Advanced Study, 87 Hoegiro, Dongdaemun-gu,
Seoul, 130-722, Republic of Korea}
\email{yhcho@kias.re.kr}
\author[M. K. Kim]{Min Kyu Kim}
\address{Department of Mathematics Education,
Gyeongin National University of Education, San 59-12, Gyesan-dong,
Gyeyang-gu, Incheon, 407-753, Republic of Korea}
\email{mkkim@kias.re.kr}

\date{\today}
\maketitle

\begin{abstract}

Let $(M,\omega)$ be a closed $2n$-dimensional symplectic manifold
equipped with a Hamiltonian $T^{n-1}$-action. Then
Atiyah-Guillemin-Sternberg convexity theorem implies that the image
of the moment map is an $(n-1)$-dimensional convex polytope. In this
paper, we show that the density function of the Duistermaat-Heckman
measure is log-concave on the image of the moment map.

\end{abstract}

\section{Introduction}

In statistic physics, the relation $S(E) = k \log{W(E)}$ is
called Boltzmann's principle where $W$ is
the number of states with given values of macroscopic parameters $E$
(like energy, temperature, ..), $k$ is the Boltzmann's constant,
and $S$ is the entropy of the system
which measures the degree of disorder in the system.
For the additive values $E$, it is well-known the entropy is
always concave function. (See \cite{O1} for more detail).
In symplectic setting, consider a Hamiltonian $G$-manifold $(M,\omega)$
with the moment map $\mu : M \rightarrow \mathfrak{g^*}$.
The Liouville measure $m_L$ is defined by
$$m_L(U) := \int_U \frac{\omega^n}{n!} $$ for any open set $U \subset M$.
Then the push-forward measure $m_{\DHfunction} := \mu_* m_L$, called
the \textit{Duistermaat-Heckman measure}, can be regarded as a
measure on $\mathfrak{g^*}$ such that for any Borel subset $B
\subset \mathfrak{g^*}$, $m_{\DHfunction}(B) = \int_{\mu^{-1}(B)}
\frac{\omega^n}{n!}$ tells us that how many states of our system
have momenta in $B.$ By the Duistermaat-Heckman theorem \cite{DH},
$m_{\DHfunction}$ can be expressed in terms of the density function
$\DHfunction(\xi)$ with respect to the Lebesque measure on
$\mathfrak{g^*}$. Therefore the concavity of the entropy of given
Hamiltonian system on $(M,\omega)$ can be interpreted as the
log-concavity of $\DHfunction(\xi)$ on the image of $\mu$. A.
Okounkov \cite{O2} proved that the density function of the
Duistermaat-Heckman measure is log-concave on the image of the
moment map for the maximal torus action when $(M,\omega)$ is the
co-adjoint orbit of some classical Lie groups. In \cite{Gr}, W.
Graham showed the log-concavity of the density function of the
Duistermaat-Heckman measure also holds for any K\"{a}hler manifold
admitting a holomorphic Hamiltonian torus action. V. Ginzberg and A.
Knudsen conjectured independently that the log-concavity holds for
any Hamiltonian $G$-manifolds, but it turns out to be false in
general by Y. Karshon \cite{K1}. Further related works can be found
in \cite{L} and \cite{C}.

As noted in \cite{K1} and \cite{Gr}, log-concavity holds for Hamiltonian toric (i.e.
complexity zero) actions, and Y. Lin dealt with log-concavity of
complexity two Hamiltonian torus actions in \cite{L}. But, there is
no result on log-concavity of complexity one Hamiltonian torus
actions. This is why we rstrict our interest to complexity one.
From now on, we assume that $(M,\omega)$ is a $2n$-dimensional
closed symplectic manifold with an effective Hamiltonian $T^{n-1}$-action.
Let $\mu : M \rightarrow \mathfrak{t}^*$ be the corresponding moment
map where $\mathfrak{t}^*$ is a dual of the Lie algebra of
$T^{n-1}$. By Atiyah-Guillemin-Sternberg convexity theorem, the
image of the moment map $\mu(M)$ is an $(n-1)$-dimensional convex
polytope in $\mathfrak{t}^*$. By the Duistermaat-Heckman theorem
\cite{DH}, we have $$ m_{\DHfunction} = \DHfunction(\xi)d\xi $$
where $d\xi$ is the Lebesque measure on $\mathfrak{t}^* \cong
\R^{n-1}$ and $\DHfunction(\xi)$ is a continuous piecewise
polynomial function of degree less than 2 on $\mathfrak{t}^*$. Our
main theorem is as follow.

\begin{theorem}\label{main}
    Let $(M,\omega)$ be a $2n$-dimensional closed symplectic manifold
    equipped with a Hamiltonian $T^{n-1}$-action with the moment map
    $\mu : M \rightarrow \mathfrak{t}^*$. Then the density function of the Duistermaat-Heckman measure is log-concave on $\mu(M)$.
\end{theorem}

\section{Proof of the theorem \ref{main}}
Let $(M, \omega)$ be a $2n$-dimensional closed symplectic manifold.
Let $(n-1)$-dimensional torus $T$ acts on $(M,\omega)$ in
Hamiltonian fashion. Denote by $\mathfrak{t}$ the Lie algebra of
$T.$ For a moment map $\mu : M \rightarrow \mathfrak{t}^*$ of the
$T$-action, define the Duistermaat-Heckman function $\DHfunction :
\mathfrak{t}^* \rightarrow \R$ as
\begin{equation*}
\DHfunction (\xi) = \int_{M_\xi} \omega_\xi
\end{equation*}
where $M_\xi$ is the reduced space $\mu^{-1}(\xi)/T$
and $\omega_\xi$ is the corresponding reduced symplectic form on $M_\xi$.

Now, we define the x-ray of our action. Let $T_1 , \cdots , T_N$ be
the subgroups of $T^{n-1}$ which occur as stabilizers of points in
$M^{2n}$. Let $M_i$ be the set of points whose stabilizers are
$T_i.$ By relabeling, we can assume that $M_i$'s are connected and
the stabilizer of points in $M_i$ is $T_i.$ Then, $M^{2n}$ is a
disjoint union of $M_i$'s. Also, it is well known that $M_i$ is open
dense in its closure and the closure is just a component of the
fixed set $M^{T_i}.$ Let $\mathfrak{M}$ be the set of $M_i$'s. Then,
the \textit{x-ray} of $(M^{2n}, \omega, \mu)$ is defined as the set
of $\mu(\overline{M_i})$'s. Here, we recall a basic lemma.
\begin{lemma}\cite[Theorem 3.6]{GS}
 \label{lemma: perpendicular subalgebra}
Let $\mathfrak{h}$ be the Lie algebra of $T_i.$ Then $\mu(M_i)$ is
locally of the form $x+\mathfrak{h}^\perp$ for some $x \in
\mathfrak{t}^*.$
\end{lemma}
By this lemma, $\dim_\R \mu(M_i) = m$ for $(n-1-m)$-dimensional
$T_i.$ Each image $\mu(\overline{M_i})$ (resp. $\mu(M_i)$) is called
an \textit{$m$-face} (resp. \textit{an open $m$-face}) of the x-ray
if $T_i$ is $(n-1-m)$-dimensional. Our interest is mainly in open
$(n-2)$-faces of the x-ray, i.e. codimension one in
$\mathfrak{t}^*.$ Figure \ref{figure: illustrate} is an example of
x-ray with $n=3$ where thick lines are $(n-2)$-faces. Now, we can
prove the main theorem.

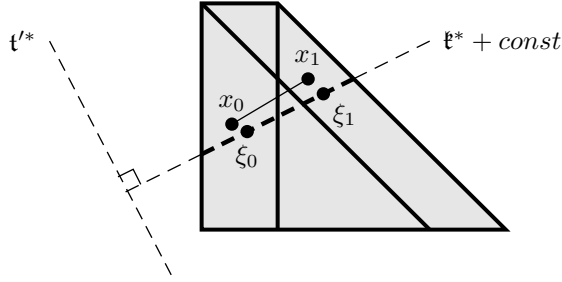
\begin{figure}[ht]
\begin{center}
\begin{pspicture}(-2, -1)(5, 3.5)\footnotesize

\pspolygon[fillstyle=solid,fillcolor=lightgray, linewidth=1.5pt](0,
0)(4, 0)(1, 3)(0, 3)(0, 0)

\psline[linewidth=1.5pt](1, 0)(1, 3)

\psline[linewidth=1.5pt](3, 0)(0, 3)

\psline[linewidth=0.5pt, linestyle=dashed](-1, 0.5)(0, 1)
\psline[linewidth=0.5pt, linestyle=dashed](-0.4, -0.6)(-2, 2.5)
\psline[linewidth=0.5pt, linestyle=dashed](2, 2)(3, 2.5)

\psline[linewidth=0.5pt](-0.8, 0.6)(-0.9, 0.8)(-1.1, 0.7)

\uput[r](3, 2.5){$\mathfrak{k}^*+const$}

\uput[l](-2, 2.5){$\mathfrak{t}^{\prime *}$}

\psline[linewidth=1.7pt, linestyle=dashed](0, 1)(2, 2)

\psline[linewidth=0.5pt](0.4, 1.4)(1.4, 2)

\psdots[dotsize=5pt](0.4, 1.4)(1.4, 2)

\uput[u](0.4, 1.4){$x_0$} \uput[u](1.4, 2){$x_1$}

\psdots[dotsize=5pt](0.6, 1.3)(1.6, 1.8)

\uput[d](0.6, 1.3){$\xi_0$}

\uput[dr](1.6, 1.8){$\xi_1$}

\end{pspicture}
\end{center}
\caption{\label{figure: illustrate} Proof of Theorem \ref{main}}
\end{figure}

\begin{proof}[Proof of Theorem \ref{main}]
When $n=2,$ we obtain a proof by \cite[Lemma 2.19]{K2}. So, we
assume $n \ge 3.$ Pick arbitrary two points $x_0, x_1$ in the image
of $\mu.$ We should show that
\begin{equation}
\label{equation: log-concave formular} t \log \big(\DHfunction (x_1)
\big)+(1-t) \log \big(\DHfunction (x_0) \big) \le \log
\big(\DHfunction ( t x_1 + (1-t) x_0 ) \big)
\end{equation}
for each $t \in [0, 1].$ Put $x_t = t x_1 + (1-t) x_0.$

Let us fix a decomposition $T=S^1 \times \cdots \times S^1.$ By the
decomposition, we identify $\mathfrak{t}$ with $\R^{n-1},$ and
$\mathfrak{t}$ carries the usual Riemannian metric $\langle ,
\rangle_0$ which is a bi-invariant metric. This metric gives the
isomorphism
\begin{equation*}
\iota : \mathfrak{t} \rightarrow \mathfrak{t}^*, ~ X \mapsto \langle
\cdot, X \rangle_0.
\end{equation*}
For a small $\epsilon
>0,$ pick two regular values $\xi_i$ in the ball $B(x_i, \epsilon)$
for $i=0, 1$ which satisfy the following two conditions:
\begin{itemize}
  \item[i.] $\xi_1-\xi_0 \in \iota(\Q^{n-1}),$
  \item[ii.] the line $L$ containing $\xi_0, \xi_1$
  in $\mathfrak{t}^*$ meets each open $m$-face
  transversely for $m=1, \cdots, n-2.$
\end{itemize}
Transversality guarantees that the line does not meet any open
$m$-face for $m \le n-3.$ Put
\begin{equation*}
\xi_t = t \xi_1 + (1-t) \xi_0 \text{ and } X = \iota^{-1} (\xi_1 -
\xi_0).
\end{equation*}
Let $\mathfrak{k} \subset \mathfrak{t}$ be the one-dimensional
subalgebra spanned by $X.$ By i., $\mathfrak{k}$ becomes a Lie
algebra of a circle subgroup of $T,$ call it $K.$ Let
$\mathfrak{t}^\prime$ be the orthogonal complement of $\mathfrak{k}$
in $\mathfrak{t}.$ Again by i., $\mathfrak{t}^\prime$ becomes a Lie
subgroup of a $(n-2)$-dimensional subtorus of $T,$ call it
$T^\prime.$ Let
\begin{equation*}
p: \mathfrak{t}^* \rightarrow \mathfrak{t}^{\prime
*} = \iota (\mathfrak{t}^\prime)
\end{equation*}
be the orthogonal projection along $\mathfrak{k}^* = \iota
(\mathfrak{k}^\prime).$ If we put $\mu^\prime = p \circ \mu,$ then
$\mu^\prime : M \rightarrow \mathfrak{t}^{\prime *}$ is a moment map
of the restriced $T^\prime$-action on $M.$ Put $\xi^\prime =
p(\xi_t)$ for $t \in [0, 1].$

We would show that $\xi^\prime$ is a regular value of $\mu^\prime.$
For this, we show that each point $x \in \mu^{\prime
-1}(\xi^\prime)$ is a regular point of $\mu^\prime.$ By ii. and
Lemma \ref{lemma: perpendicular subalgebra}, stabilizer $T_x$ is
finite or one-dimensional. If $T_x$ is finite, then $x$ is a regular
point of $\mu$ so that it is also a regular point of $\mu^\prime.$
If $T_x$ is one-dimensional, then $\mu(x)$ is a point of an open
$(n-2)$-face $\mu(M_i)$ such that $x \in M_i.$ Let $\mathfrak{h}$ be
the Lie algebra of $T_i = T_x.$ By Lemma \ref{lemma: perpendicular
subalgebra}, $p(d\mu(T_x M_i))=p(\mathfrak{h}^\perp),$ and the
kernel $\mathfrak{k}$ of $p$ is not contained in
$\mathfrak{h}^\perp$ by transversality. So, $p(\mathfrak{h}^\perp)$
is the whole $\mathfrak{t}^{\prime *}$ because $\dim
\mathfrak{h}^\perp = \dim \mathfrak{t}^{\prime *},$ and this means
that $x$ is a regular point of $\mu^\prime.$ Therefore, we have
shown that $\xi^\prime$ is a regular value of $\mu^\prime.$

Since $\xi^\prime$ is a regular value, the preimage $\mu^{\prime
-1}(\xi^\prime)$ is a manifold  and $T^\prime$ acts almost freely on
it, i.e. stabilizers are finite. So, if we denote by
$M_{\xi^\prime}$ the symplectic reduction $\mu^{\prime
-1}(\xi^\prime)/T^\prime,$ then it becomes a symplectic orbifold
carrying the induced symplectic $T/T^\prime$-action. We can observe
that the image of $\mu^{\prime -1}(\xi^\prime)$ through $\mu$ is the
thick dashed line in Figure \ref{figure: illustrate}. Since $K/(K
\cap T^\prime) \cong T/T^\prime,$ we will regard $K/(K \cap
T^\prime)$ and $\mathfrak{k}$ as $T/T^\prime$ and its Lie algebra,
respectively. The map $\mu_X := \langle \mu, X \rangle$ induces a
map on $M_{\xi^\prime}$ by $T$-invariance of $\mu,$ call it just
$\mu_X$ where $\langle ~, ~ \rangle : \mathfrak{t}^* \times
\mathfrak{t} \rightarrow \R$ is the evaluation pairing. Then, we can
observe that $\mu_X$ is a Hamiltonian of the $K/(K \cap
T^\prime)$-action on $M_{\xi^\prime},$ and that $M_{\xi_t}$ is
symplectomorphic to the symplectic reduction of $M_{\xi^\prime}$ at
the regular value $\langle \xi_t, X \rangle$ with respect to
$\mu_X.$ If we denote by $\DHfunction_X$ the Duistermaat-Heckman
function of $\mu_X : M_{\xi^\prime} \rightarrow \R,$ then we have
$\DHfunction (\xi_t) = \DHfunction_X (\langle \xi_t, X \rangle)$ for
$t \in [0, 1].$ Since $M_{\xi^\prime}$ is a four-dimensional
symplectic orbifold with Hamiltonian circle action, $\DHfunction_X$
is log-concave by Lemma \ref{orbifoldlogconcave} below. Since $x_t$
and $\xi_t$ are sufficiently close and $\DHfunction$ is continuous
by \cite{DH}, we can show (\ref{equation: log-concave formular}) by
log-concavity of $\DHfunction_X.$
\end{proof}

\begin{lemma}\label{orbifoldlogconcave}
    Let $(N,\sigma)$ be a closed four dimensional Hamiltonian $S^1$-orbifold. Then the density function of the Duistermaat-Heckman measure is
    log-concave.
\end{lemma}

\begin{proof}
    Let $\phi : N \rightarrow \R$ be a moment map. Then the density function $\DHfunction : \textrm{Im}\phi \rightarrow \R_{\geq 0}$ of the Duistermaat-Heckman measure is given by
    $$ \DHfunction(t) = \int_{N_t} \sigma_t $$
    for any regular value $t \in \textrm{Im} \phi$. Let $(a,b) \subset \textrm{Im} \phi$ be an open interval consisting of regular values of $\phi$ and fix $t_0 \in (a,b)$.
    By the Duistermaat-Heckman theorem \cite{DH}, $[\sigma_t] - [\sigma_{t_0}] = -e(t-t_0)$ for any $t \in (a,b)$, where $e$ is the Euler class of the $S^1$-fibration $\phi^{-1}(t_0) \rightarrow \phi^{-1}(t_0) / S^1$. Therefore
    $$ DH'(t) = - \int_{N_t} e $$ and $$ DH''(t) = 0 $$ for any $t \in (a,b)$. Note that $\DHfunction(t)$ is log-concave on $(a,b)$ if and only if it satisfies  $\DHfunction(t)\cdot \DHfunction''(t) - \DHfunction'(t)^2 \leq 0$ for all $t \in (a,b)$. Hence $\DHfunction(t)$ is log-concave on any open intervals consisting of regular values.

    Let $c$ be any interior critical value of $\phi$ in $\textrm{Im}\phi$. Then it is enough to show that the jump in the derivative of $(\log{\DHfunction})'$ is negative at $c$.
    First, we will show that the jump of the value $ \DHfunction'(t) = - \int_{N_t} e $ is negative at $c$. Choose a small $\epsilon > 0$ such that
    $(c-\epsilon, c+\epsilon)$ does not contain a critical value except for $c$.
    Let $N_c$ be a symplectic cut of $\phi^{-1}[c-\epsilon,c+\epsilon]$ along the extremum so that $N_c$ becomes a closed Hamiltonian $S^1$-orbifold whose maximum is the reduced space $M_{c+\epsilon}$ and the minimum is $N_{c-\epsilon}$.
    Using the Atiyah-Bott-Berline-Vergne localization formula for orbifolds \cite{M}, we have

    $$0 = \int_{N_c} 1 = \sum_{p \in N^{S^1} \cap \phi^{-1}(c)} \frac{1}{d_p} \frac{1}{p_1p_2 \lambda^2} + \int_{M_{c-\epsilon}} \frac {1}{\lambda + e_-}  + \int_{N_{c+\epsilon}} \frac {1}{- \lambda - e_+}$$
    which is equivalent to
    $$ 0 = \sum_{p \in N^{S^1} \cap \phi^{-1}(c)} \frac{1}{p_1p_2} = \int_{N_{c-\epsilon}} e_- - \int_{N_{c+\epsilon}}e_+,$$
    where $d_p$ is the order of the local group of $p$, $p_1$ and $p_2$ are the weights of the tangential $S^1$-representation on $T_pN$, and
    $e_-$ ($e_+$ respectively) is the Euler class of $\phi^{-1}(c-\epsilon)$ ($\phi^{-1}(c+\epsilon)$ respectively). Since $c$ is in the interior of
    $\textrm{Im}\phi$, we have $p_1p_2 < 0 $ for any $p \in N^{S^1} \cap \phi^{-1}(c)$. Hence the jump of $ \DHfunction'(t) = - \int_{N_t} e $ is negative
    at $c$, which implies that the jump of $\log{\DHfunction(t)}' = \frac{\DHfunction'(t)}{\DHfunction(t)}$ is negative at $c$ (by continuity of $\DHfunction(t)$). It finishes the proof.
\end{proof}

\bigskip
\bibliographystyle{amsalpha}

\end{document}